\documentclass[11pt,reqno]{amsart}

\usepackage[l2tabu,orthodox]{nag}

\usepackage[all]{xy}
\usepackage{xcolor}

\usepackage{hyperref}

\usepackage{marginnote}
\usepackage{enumerate}

\usepackage[swedish,english]{babel}
\usepackage[T1]{fontenc}
\usepackage[latin1]{inputenc}

\usepackage{amsaddr}

\usepackage{amsmath,amssymb,amsthm}

\theoremstyle{plain}
\newtheorem{theorem}{Theorem}
\newtheorem{lemma}[theorem]{Lemma}
\newtheorem{prop}[theorem]{Proposition}

\theoremstyle{definition}
\newtheorem{defn}[theorem]{Definition}

\theoremstyle{remark}
\newtheorem{remark}[theorem]{Remark}

\newcommand{\N}{\mathbb{N}}
\newcommand{\Z}{\mathbb{Z}}

\renewcommand{\emptyset}{\varnothing}
\newcommand{\LL}{\mathcal{L}}

\DeclareMathOperator{\Supp}{Supp}

\title[Simplicity of Leavitt path algebras]{Simplicity 
of Leavitt path algebras via graded ring theory}
\date{\today}

\begin{document}

\author{Patrik Lundstr\"{o}m}
\address{Department of Engineering Science,\\
University West, SE-46186 Trollh\"{a}ttan, Sweden}
\email{patrik.lundstrom@hv.se}

\author{Johan \"{O}inert}
\address{Department of Mathematics and Natural Sciences,\\
Blekinge Institute of Technology,
SE-37179 Karlskrona, Sweden}
\email{johan.oinert@bth.se}

\subjclass[2020]{16S88, 16W50}
\keywords{Leavitt path algebra, graded ring, simple ring, center}

\begin{abstract}
Suppose that $R$ is an associative unital
ring and that 
$E=(E^0,E^1,r,s)$ is a directed graph.
Utilizing results from graded ring theory
we show, that the associated Leavitt path algebra 
$L_R(E)$ is simple if and only if $R$ is 
simple, 
$E^0$ has no nontrivial hereditary and saturated subset, 
and 
every 
cycle 
in $E$ has an exit.
We also give a complete description of the center of a simple Leavitt path algebra.
\end{abstract}

\maketitle


\section{Introduction}

The Leavitt path algebra of a row-finite graph, over a field, was introduced in \cite{AbramsAP2005,AraMorenoPardo}
and has since then been successively generalized (see e.g \cite{AbramsAP2008,Tomforde2011}).
The Leavitt path algebra of an arbitrary directed graph, over a unital ring, was introduced in \cite{H2013A}.
For an
account of the development of the field of Leavitt path algebras, we refer the reader to \cite{AbramsDecade}.
Here is our first main result.

\begin{theorem}\label{thm:main}
Suppose that $R$ is an associative unital ring 
and that $E=(E^0,E^1,r,s)$ is a directed graph.
The Leavitt path algebra $L_R(E)$ is simple if and only if 
$R$ is simple, 
$E^0$ has no nontrivial hereditary and saturated subset, 
and 
every 
cycle in $E$ has an exit.
\end{theorem}

Characterizations of simple Leavitt path algebras over fields have previously been established in e.g. \cite[Thm.~6.18]{Tomforde2007}, 
\cite[Thm.~3.1]{AbramsAP2008}
and
\cite[Thm.~3.5]{GoncalvesOinertRoyer}.
%
Theorem~\ref{thm:main} generalizes all of those results,
and also partially generalizes  \cite[Thm.~7.20]{Tomforde2011}.
Our second main result, stated below, completely describes the center of a simple Leavitt path algebra. It generalizes 
\cite[Thm.~4.2]{APCrow2011} from the case where $R$ is a field and $E$ is a row-finite graph.

\begin{theorem}\label{thm:center}
Suppose that $R$ is an associative unital ring 
and that $E=(E^0,E^1,r,s)$ is a directed graph.
Furthermore, suppose that $L_R(E)$ is a simple Leavitt path algebra. 
The following assertions hold:
\begin{enumerate}[{\rm (a)}]
    \item If
$L_R(E)$ is not unital, then $Z(L_R(E)) = \{0\}$.

    \item If $L_R(E)$ is unital, then $Z(L_R(E)) = Z(R) \cdot 1_{L_R(E)}$.
\end{enumerate}
\end{theorem}

Whereas earlier proofs of Theorems~\ref{thm:main} and \ref{thm:center} 
(when $R$ is a field) utilize \emph{ad hoc}
arguments, specifically designed for graph algebras, we use the general theory of 
\emph{graded rings} to obtain our results.
This makes our proofs shorter and, we believe,
clearer. Indeed, we show that $L_R(E)$ is
graded simple if and only if $R$ is simple and
$E^0$ has no nontrivial hereditary and
saturated subset (see Proposition~\ref{prop:GradedSimple}).
We also show that every cycle
in $E$ has an exit if and only if the center
of each corner subring of $L_R(E)$ at a vertex has 
degree zero (see Proposition~\ref{prop:center}).

We point out that there are various generalizations of Leavitt path
algebras in the literature (see e.g. \cite[Sec.~5]{AbramsDecade} and \cite{ClarkEtAl2014,Steinberg2010}). A simplicity result for Steinberg algebras was obtained in
\cite{ClarkEdieMichell2015}, and when translated to Leavitt path algebras one recovers Theorem~\ref{thm:main} in the special case where $R$ is a
commutative unital ring.
Note that \cite[Thm.~4.2]{APCrow2011} was generalized to Kumjian-Pask
algebras in \cite{BrownHuef2014},
and in \cite{CMBMGSM2019}, Steinberg algebra techniques were used to give a
complete description of the center of a general (not necessarily simple)
Leavitt path algebra $L_R(E)$, where $R$ is a commutative unital ring.

\section{Simple $\Z$-graded rings}\label{sec:gradedrings}

Let $\Z$ denote the rational integers and
write $\N := \{ 1,2,3,\ldots \}$.
Suppose that $S$ is a ring. 
By this we mean that
$S$ is associative but not necessarily 
unital. If $S$ is unital, then we let $1_S$ 
denote the multiplicative identity of $S$.
Furthermore, we let $Z(S)$ denote the 
center of $S$, that is the set of all 
$s \in S$ satisfying $st = ts$ for every $t \in S$.
Recall that $S$ is said to be 
\emph{$\Z$-graded} if, for each $n \in \Z$, 
there is an additive subgroup $S_n$ of $S$
such that $S = \oplus_{n\in \Z} S_n$,
and 
$S_n S_m \subseteq S_{n+m}$,
for all $n,m\in \Z$.
In that case, each element $s\in S$ may be written as $s=\sum_{n\in \Z} s_n$, where $s_n \in S_n$ is zero for all but finitely many $n\in \Z$.
The \emph{support of $s$} is defined as the finite set $\Supp(s):=\{n\in \Z \mid s_n \neq \{0\}\}$.
An ideal $I$ of $\Z$-graded ring $S$ is said to be \emph{graded}, if $I=\oplus_{n\in \Z} (I \cap S_n)$.
If $\{0\}$ and $S$ are the only graded ideals of $S$, then $S$ is said to be \emph{graded simple}.

We recall some properties of graded rings:

\begin{lemma}\label{lem:pregraded}
Suppose that $S$ is a unital $\Z$-graded ring.
\begin{itemize}
    
    \item[(a)] The ring $Z(S)$ is $\Z$-graded with 
    respect to the grading defined by
    $Z(S)_n := Z(S) \cap S_n$,
    for $n \in \Z$.
    
    \item[(b)] If $S$ is a field, then $S = S_0$.
    
\end{itemize}
\end{lemma}

\begin{proof}
(a) is \cite[p.~15, Exer.~8]{nastasescu2004}
and 
(b) is \cite[Rem.~1.3.10]{nastasescu2004}. 
\end{proof}

Next, we state a special case of
\cite[Thm.~1.2]{NO2015}
and \cite[Thm.~5]{Jespers1989}. For the 
convenience of the reader, we include 
a shortened version of the proof given 
in loc.~cit. adapted 
to the situation at hand.

\begin{prop}\label{prop:simplicity}
Suppose that $S$ is a unital $\Z$-graded ring.
Then the following assertions are equivalent:
\begin{itemize}
    \item[(i)] $S$ is simple;
    
    \item[(ii)] $S$ is graded simple and 
    $Z(S)$ is a field;
    
    \item[(iii)] $S$ is graded simple and 
    $Z(S) \subseteq S_0$.
    
\end{itemize}
\end{prop}

\begin{proof}
(i)$\Rightarrow$(ii) is 
clear, and (ii)$\Rightarrow$(iii) follows from 
Lemma~\ref{lem:pregraded}.
Now, we show that (iii)$\Rightarrow$(i).
Suppose that $S$ is graded simple and that
$Z(S) \subseteq S_0$. 
Let $I$ be a nonzero ideal
of $S$. We wish to show that $1_S \in I$.
Amongst all nonzero elements of $I$, choose $s$ such that $|\Supp(s)|$ is minimal. 
Take $m \in \Supp(s)$. 
Since $S$ is graded simple, there are $n \in \N$
and homogeneous elements $p_1,\ldots,p_n,q_1,\ldots,q_n \in S$, 
such that 
$\sum_{i=1}^n p_i s_m q_i = 1_S$, and
 $p_i s_m q_i \in S_0 \setminus \{0\}$ for every $i\in \{1,\ldots,n\}$.
Write $t := \sum_{i=1}^n p_i s q_i$.
Note that $t \in I$, $t_0 = 1_S$ and 
$|\Supp(t)| \leq |\Supp(s)|$.
Take $z \in \Z$ and $x \in S_z$.
Then, $tx-xt \in I$ and,
since $t_0 = 1_S$, it follows that
$|\Supp(tx - xt)| < |\Supp(t)|$.
By the assumptions on $s$ we get
$|\Supp(tx - xt)| = 0$ and hence that
$xt = tx$. Thus, $t \in Z(S) \subseteq S_0$.
We conclude that $1_S = t_0 = t \in I$.
\end{proof}

Let $S$ be a ring.
Recall from \cite{anh1987} 
(see also \cite{nystedt2019})
that a set $U$ of idempotents in $S$ is
called a
\emph{set of local units} for $S$, if
for every $n \in \N$ and all 
$s_1,\ldots,s_n \in S$ there is some $e \in U$
such that $e s_i = s_i e = s_i$, for every $i \in \{1,\ldots,n\}$. 

\begin{remark}
Suppose that $S$ is a $\Z$-graded ring.
If $e\in S_0$ is an idempotent, then the corner subring $eSe$ inherits a natural $\Z$-grading defined by $(eSe)_n:= e S_n e$, for $n\in \Z$.
\end{remark}

For future reference, we 
recall the following two results:

\begin{prop}\label{prop:simplicitylocal}
Suppose that $S$ is a $\Z$-graded ring 
equipped with a set of local units $U \subseteq S_0$. Then, $S$ is (graded)
simple if and only if, for every $f \in U$, the ring $f S f$ is (graded) simple.
\end{prop}

\begin{proof}
First we show the ``only if''
statement. Suppose that $S$ is (graded) simple 
and that $f \in U$. 
Let $J$ be a nonzero (graded) ideal of $f S f$.
By (graded) simplicity of $S$, it follows 
that $SJS = S$.
Thus, $fSf = fSJSf = (fSf)J(fSf) \subseteq J$
and hence $J = fSf$.
Now, we show the ``if''
statement. Suppose that $fSf$ is (graded) simple
for every $f \in U$. Let $I$ be a nonzero 
(graded) ideal
of $S$. Take a nonzero (homogeneous) $x \in S$. 
Take a nonzero (homogeneous) $y \in I$ and $f \in U$
with $fx = xf = x$ and $fy = yf = y$. 
By (graded) simplicity of $fSf$
it follows that $I \supseteq f SyS f = fSf \ni x$.
Thus, $I = S$.
\end{proof}

\begin{prop}\label{prop:simplicitylocalagain}
Suppose that $S$ is a $\Z$-graded ring 
equipped with a set of local units
and that
$f \in S_0$ is a nonzero idempotent.
If $S$ is graded simple and $f S f$ is 
simple, then $S$ is simple.
\end{prop}

\begin{proof}
Suppose that $S$ is graded simple and that $f S f$ is simple.
Let $I$ be a nonzero ideal of $S$. 
Take a nonzero $s \in I$ and write
$s = \sum_{n \in \Supp(s)} s_n$.
Fix $m \in \Supp(s)$ and define $J := S s_m S$.
Then, $J$ is a nonzero graded ideal of $S$.
By graded simplicity of $S$, it follows that
$J = S$ and, in particular, that $f \in J$.
Note that $f \in fJf$.
Using that $f \neq 0$, it follows
that there exist nonzero homogeneous
$y,z \in S$ such that $fy s_m zf$ is
nonzero and $\deg(y)+\deg(z) = -m$.
Now, define $s' := fy s zf$. By the construction
of $s'$, it follows that $s' \in I \cap fSf$
and that $s'$ is nonzero. In particular,
$I \cap fSf \neq \{ 0 \}$. Hence, by simplicity of $fSf$, we get that $I \cap fSf = fSf$.
Thus, $f \in I$. 
Note that $SfS$ is a nonzero graded ideal of $S$. Hence, by graded simplicity of $S$, we get that
$I \supseteq SfS = S$. This shows that $I = S$.
\end{proof}

\section{Simple Leavitt path
algebras}\label{sec:LPA}

Let $R$ be an associative unital ring and let
$E = (E^0,E^1,r,s)$ be a directed graph.
Recall that 
$r$ (range) and $s$ (source) are maps $E^1 \to E^0$. The elements of $E^0$ are called \emph{vertices} 
and the elements of $E^1$ are called \emph{edges}. 
The elements of $E^1$ are called \emph{real edges}, while for $f\in E^1$ 
we call $f^*$ a \emph{ghost edge}.
The set $\{f^* \mid f \in E^1\}$ will be denoted by $(E^1)^*$.
A \emph{path} $\mu$ in $E$ is a sequence of edges 
$\mu = \mu_1 \ldots \mu_n$ such that $r(\mu_i)=s(\mu_{i+1})$ 
for $i\in \{1,\ldots,n-1\}$. In that case, $s(\mu):=s(\mu_1)$ 
is the \emph{source} of $\mu$, $r(\mu):=r(\mu_n)$ is the \emph{range} 
of $\mu$, 
and $|\mu|:=n$ is the \emph{length} of $\mu$.
If $\mu=\mu_1\ldots \mu_n$ is a (real) path in $E$, then we let $\mu^*:=\mu_n^*\ldots \mu_1^*$ denote the corresponding \emph{ghost path}.
For any vertex $v \in E^0$ we put $s(v):=v$ and $r(v):=v$.
We let $r(f^*)$ denote $s(f)$, and we let $s(f^*)$ denote $r(f)$.
For $n \geq 2$, we define $E^n$ to be the set of paths of length $n$, and
$E^* := \cup_{n\geq 0} E^n$ is the set of all finite paths.

Following Hazrat \cite{H2013A} we make the following definition.

\begin{defn}\label{def:LPA}
The \emph{Leavitt path algebra of $E$ with coefficients in $R$}, denoted by $L_R(E)$,
is the algebra generated by the sets
$\{v \mid v\in E^0\}$, $\{f \mid f\in E^1\}$ and $\{f^* \mid f\in E^1\}$
with the coefficients in $R$,
subject to the relations:
\begin{enumerate}
	\item $uv = \delta_{u,v} v$ for all $u,v \in E^0$;
	\item $s(f)f=fr(f)=f$ and $r(f)f^*=f^*s(f)=f^*$, for all $f\in E^1$;
	\item $f^*f'=\delta_{f,f'} r(f)$, for all $f,f'\in E^1$;
	\item $\sum_{f\in E^1, s(f)=v} ff^* = v$, for every $v\in E^0$ for which $s^{-1}(v)$ is non-empty and finite.
\end{enumerate}
Here elements of the ring $R$ commutes with the generators.
\end{defn}

\begin{remark}\label{rem:Zgradation}
(a)
The Leavitt path algebra $L_R(E)$ carries a natural $\Z$-gradation.
Indeed, put $\deg(v):=0$ for each $v\in E^0$.
For each $f\in E^1$ we put $\deg(f):=1$ and $\deg(f^*):=-1$.
By assigning degrees to the generators in this way,
we obtain a $\Z$-gradation on the free algebra $F_R(E) = R \langle	v, f, f^* \mid v \in E^0, f \in E^1 \rangle$.
Moreover, the ideal coming from relations (1)--(4) in Definition~\ref{def:LPA} is
graded. 
Using this it is easy to see that the natural $\Z$-gradation on $F_R(E)$ carries over to a $\Z$-gradation on the quotient algebra $L_R(E)$.

(b)
The set $\left\{\sum_{v\in F} v \mid F \text{ is a finite subset of } E^0\right\}$
is a set of local units for $L_R(E)$.
If $E^0$ is finite,
then $L_R(E)$ is unital, and $1_{L_R(E)}=\sum_{v\in E^0} v$.

(c) Motivated by Definition~\ref{def:LPA} (2),
for $u\in E^0$, we write $u^*:=u$.
\end{remark}

\begin{defn}
Let $E=(E^0,E^1,r,s)$ be a directed graph.
A subset $H \subseteq E^0$ is said to be \emph{hereditary}
if, for any $f\in E^1$, we have that $s(f) \in H$
implies $r(f) \in H$.
A hereditary subset $H \subseteq E^0$ is called \emph{saturated} if,
whenever $v\in E^0$ satisfies $0 < |s^{-1}(v)| < \infty$,
we have that $\{r(f) \in H \mid f \in E^1 \text{ and } s(f)=v \} \subseteq H$ implies $v\in H$.
\end{defn}

\begin{remark}
Note that $\emptyset$ and $E^0$ are always hereditary and saturated subsets of $E^0$. They are referred to as \emph{trivial}.
\end{remark}

\begin{prop}\label{prop:GradedSimple}
The Leavitt path algebra $L_R(E)$ is graded simple if and
only if $R$ is simple and 
$E^0$ has no nontrivial hereditary and saturated subset.
\end{prop}

\begin{proof}
First we show the ``if'' statement.
Suppose that $R$ is simple and that $E^0$ has no  
nontrivial
hereditary and saturated subset.
Let $I$ be a nonzero graded ideal of $L_R(E)$.
Consider the set $H_I := \{ v \in E^0 \mid k v \in I 
\text{ for some nonzero } k \in R \}$.
By the same argument as in \cite[Lem.~5.1]{Tomforde2011}, 
$H_I$
is nonempty. Furthermore, since $R$ is simple, it follows that
$H_I = \{ v \in E^0 \mid v \in I \}$.
We wish to show that $H_I$ is
hereditary and saturated. To this end, take 
$v\in H_I$. 
Suppose that $e \in E^1$ with $s(e) = v$.
Then, $r(e) = e^* e = e^* v e \in I$. 
Thus, $H_I$ is hereditary.
Now, take 
$v\in E^0$ such that $0 < |s^{-1}(v)| < \infty$,
and suppose that $r(s^{-1}(v)) \subseteq H_I$.
For each $e \in s^{-1}(v)$ we have $r(e) \in H_I$
and hence $e e^* = e r(e) e^* \in I$. Thus, 
$v = \sum_{e \in s^{-1}(v)} ee^* \in I$ and $v \in H_I$.
Therefore, $H_I$ is saturated. 
By assumption, we get that $H_I = E^0$.
This shows that $I$ must contain all the local units 
of $L_R(E)$ and thus $I = L_R(E)$. Hence, $L_R(E)$ is 
graded simple.

Now, we show the ``only if'' statement.
Suppose that $L_R(E)$ is graded simple.
Let $J$ be a nonzero ideal of $R$.
Note that $J \cdot L_R(E)$ is a nonzero graded ideal of $L_R(E)$.
Thus, $J \cdot L_R(E) = L_R(E)$ and we conclude that $J=R$. This shows that $R$ is simple.

Let $H$ be a proper hereditary and saturated subset of $E^0$. 
Following \cite{AbramsAP2005,AbramsAP2008}, we 
let $F:=(F^0,F^1,r,s)$ be the graph consisting of all
vertices not in $H$ and all egdes whose
range is not in $H$. 
For $v \in E^0$, define $\Psi(v) := v$, if $v \in F^0$, and
$\Psi(v) := 0$, otherwise. For $e \in E^1$, define $\Psi(e) := e$, if $e \in F^1$, and
$\Psi(e) := 0$, otherwise. Furthermore, define
$\Psi(e^*) := e^*$, if $e^* \in (F^1)^*$,
and $\Psi(e^*) := 0$, otherwise. The argument
in loc. cit. shows that this yields a well-defined ring homomorphism 
$\Psi : L_R(E) \to L_R(F)$. Clearly, $\Psi$ is
\emph{graded}. Thus, the ideal 
$I := \ker(\Psi)$ of $L_R(E)$ is  
\emph{graded}.
Note that $F^0$ is nonempty, because $H$ is proper, and hence $I \neq L_R(E)$.
By assumption, we get that
$I=\{0\}$.
By the construction of $\Psi$ it follows
that $H \subseteq I$. 
Thus, $H=\emptyset$.
\end{proof}

\begin{defn}\label{def:star}
Define an additive map $\LL : L_R(E) \to L_R(E)$ by requiring that 
$\LL(\lambda \alpha \beta^*) = 
\lambda \beta \alpha^*$, for all
$\lambda \in R$, and $\alpha,\beta \in E^*$.
\end{defn}

\begin{remark}
The map $\LL$ is an 
isomorphism of additive groups such that
$\LL((L_R(E))_N) = (L_R(E))_{-N}$ for every
$N \in \Z$.
\end{remark}

\begin{lemma}\label{lem:star}
Suppose that $u \in E^0$.
The map $\LL$ 
restricts to an isomorphism of additive groups 
$\LL\lvert_{Z(u L_R(E) u)} : Z(u L_R(E) u) \to Z(u L_R(E) u)$.
In particular, the equality 
$\LL((Z(u L_R(E) u))_N) =
(Z(u L_R(E) u))_{-N}$ holds for every $N \in \Z$.
\end{lemma}

\begin{proof}
Let $x = \sum_{j=1}^m \lambda_j
\alpha_j \beta_j^* \in Z(u L_R(E) u)$,
where $\lambda_j \in R$,
$\alpha_j,\beta_j \in E^*$
and $s(\alpha_j)=s(\beta_j)=u$
for $j \in \{1,\ldots,m\}$. Take $r \in R$.
Then, 
$0 = xru-ru x = 
\sum_{j=1}^m
(\lambda_j r - r \lambda_j) \alpha_j
\beta_j^*$. Therefore, $0 = \LL(0) = 
\sum_{j=1}^m (\lambda_j r - r \lambda_j) \LL(\alpha_j \beta_j^*) = 
\sum_{j=1}^m ( \lambda_j r - r \lambda_j )
\beta_j \alpha_j^* = \LL(x) ru - ru \LL(x)$.
Thus, $\LL(x) ru = ru \LL(x)$. Take 
$\gamma,\delta \in E^*$ with $s(\gamma) =
s(\delta) = u$. Then, $0 = x \gamma \delta^*
- \gamma \delta^* x = 
\sum_{j=1}^m \lambda_j (\alpha_j \beta_j^*
\gamma \delta^* - \gamma \delta^* \alpha_j
\beta_j^*)$. Therefore, $0 = \LL(0) = 
\sum_{j=1}^m \lambda_j \LL(\alpha_j \beta_j^*
\gamma \delta^* - \gamma \delta^* \alpha_j
\beta_j^*) = 
\sum_{j=1}^m \lambda_j (\delta\gamma^*
\beta_j\alpha_j^* - \beta_j\alpha_j^*
\delta\gamma^*) = \delta\gamma^* \LL(x) -
\LL(x) \delta\gamma^*$. Thus, $\LL(x)\delta\gamma^*
= \delta\gamma^* \LL(x)$. Finally, 
$\LL(x) r \delta\gamma^* = \LL(x) ru \delta\gamma^*
= ru \LL(x) \delta\gamma^* = ru \delta\gamma^*
\LL(x) = r \delta\gamma^* \LL(x)$. 
This shows that $\LL(x) \in Z(u L_R(E) u)$.
\end{proof}

\begin{lemma}\label{lem:Zinvariant}
Suppose that $u,v \in E^0$ and 
that $\alpha \in E^*$ 
is such that 
$s(\alpha) = u$
and $r(\alpha) = v$. If $x \in Z(u L_R(E) u)$,
then $\alpha^* x \alpha \in Z(v L_R(E) v)$.
\end{lemma}

\begin{proof}
Let $x \in Z(u L_R(E) u)$.
Take $y \in v L_R(E) v$. Since $\alpha y \alpha^* \in u L_R(E) u$,
it follows that
$y \alpha^* x \alpha = v y \alpha^* x \alpha 
= \alpha^* \alpha y \alpha^* x \alpha
= \alpha^* x \alpha y \alpha^* \alpha 
= \alpha^* x \alpha y v
= \alpha^* x \alpha y$. 
Thus, $\alpha^* x \alpha \in Z(v L_R(E) v)$.
\end{proof}

\begin{defn}[cf.~\cite{Tomforde2011}]
Let $E=(E^0,E^1,r,s)$ be a directed graph.
A \emph{cycle} in $E$ is a path $\mu \in E^*\setminus E^0$ 
such that $s(\mu)=r(\mu)$.
An edge $f \in E^1$ is said to be an \emph{exit} for the 
cycle
$\mu=\mu_1 \ldots \mu_n$ if,
for some $i \in \{1, 2, \ldots, n\}$, we have $s(f)=s(\mu_i)$ but $f \neq \mu_i$.
\end{defn}

\begin{remark}
The definition of a \emph{cycle} in a directed graph varies in the literature on Leavitt path algebras.
In contrast to the most common definition of a cycle (cf.~\cite[p.~320]{AbramsAP2005}, following \cite{Tomforde2011} we allow a cycle to ``intersect'' itself.
In Theorem~\ref{thm:main}, the condition that ``every cycle in $E$ has an exit'' appears. That condition is commonly known as \emph{Condition (L)}.
It is easy to see that Condition (L) is satisfied with the first definition of a cycle \cite{AbramsAP2005}, if and only if it is satisfied with the second definition of a cycle \cite{Tomforde2011}.
\end{remark}

\begin{lemma}\label{lem:independent}
Every element in 
$E^0 \cup E^1 \cup (E^1)^*$ is nonzero
in $L_R(E)$, and
the set  of real (resp. ghost) paths
is linearly independent in the left $R$-module
$L_R(E)$ and in the right $R$-module $L_R(E)$.
\end{lemma}

\begin{proof}
The proof of
\cite[Prop.~4.9]{Tomforde2011} immediately carries over to
the case where $R$ is a noncommutative unital ring.
The same holds for the proof of 
\cite[Prop.~3.4]{Tomforde2011}
in case $E^0$ and $E^1$ are countable sets.
Otherwise, the proof may be adapted by taking $\aleph$ to be an infinite cardinal at least as large as $\text{card}(E^0 \cup E^1)$ and defining $Z:=\oplus_{\aleph} R$ (with the notation of \cite[Prop.~3.4]{Tomforde2011}).
\end{proof}

\begin{remark}
Let $x$ be a nonzero element of $L_R(E)$.
It is clear from the definition of $L_R(E)$ 
that $x$ can be
represented as a finite sum 
$x = \sum_{i=1}^n r_i \alpha_i
\beta_i^*$ where $r_i \in R \setminus \{ 0 \}$
and $\alpha_i,\beta_i \in E^*$.
Following \cite[Def.~4.8]{Tomforde2011},
we define the \emph{real degree} (resp. \emph{ghost degree}) of this representation
as $\max \{ \deg(\alpha_i)
\mid 1 \leq i \leq n \}$ 
(resp. $\max \{ \deg(\beta_i) \mid 
1 \leq i \leq n \}$). Note that, in general,
the real degree and ghost degree of $x$ depend on the 
particular choice of representation.
If, however, $x$ has a representation in
only real (resp. ghost) edges, that is if 
$x = \sum_{i=1}^n r_i \alpha_i$ 
(resp. $x = \sum_{i=1}^n r_i \beta_i^*$), 
then, by Lemma~\ref{lem:independent}, the real (resp. ghost) degree is 
independent of the choice of 
representation of $x$ in real (resp. ghost)
edges.
\end{remark}

\begin{prop}\label{prop:center}
Every cycle in $E$ has an exit 
if and only if for every $u \in E^0$ the inclusion 
$Z(u L_R(E) u) \subseteq (u L_R(E) u)_0$ holds.
\end{prop}

\begin{proof}
First we show the ``if'' statement by contrapositivity.
Suppose that 
there is a cycle 
$p \in E^* \setminus E^0$ 
without any exit. 
Set $u:=s(p)$ and write $p^0:=u$.
Take $r \in R$ and $\alpha,\beta \in E^*$ with 
$s(\alpha) = s(\beta) = u$
and $r(\alpha)=r(\beta)$.
Since $p$ has no exit, there are $m,n \in \N \cup \{0\}$
and $\gamma \in E^*$ such that $\alpha = p^m \gamma$
and $\beta = p^n \gamma$. 
Note that $\gamma\gamma^*=u=pp^*$.
We get that 
$p r\alpha \beta^* = p r p^m \gamma \gamma^* (p^*)^n =
r p^{m+1} (p^*)^n$ 
and 
$r\alpha \beta^* p = r p^m \gamma \gamma^* (p^*)^n p =
r p^m (p^*)^n p$. 
If $n=0$, then $p^{m+1}(p^*)^n=p^{m+1}= 
p^m (p^*)^n p$, and if $n>0$, then
$p^{m+1}(p^*)^n = p^m pp^* (p^*)^{n-1} = p^m (p^*)^{n-1} p^*p = 
p^m (p^*)^{n} p$.
In either case, we get that
$p r\alpha \beta^* = r\alpha \beta^* p$. 
Thus, $p \in Z(u L_R(E) u) \setminus (u L_R(E) u)_0$.

Now we show the ``only if'' statement.
Suppose that every cycle in $E$
has an exit. 
Take $u \in E^0$.
We wish to show that $Z(u L_R(E) u) \subseteq (u L_R(E) u)_0$. 
By Lemma~\ref{lem:pregraded}(a)
and Lemma \ref{lem:star}, it is enough to show that 
$(Z(u L_R(E) u))_N = \{ 0 \}$ for every negative integer $N$.

We now adapt parts of the proof of
\cite[Thm.~3.1]{AbramsAP2008} to our 
situation. Take $N < 0$.
Seeking a contradiction, suppose that 
the set
\[ M:= \{ (u,x) \mid u \in E^0 \text{ and } 
\
x \in (Z(uL_R(E)u))_N \setminus \{ 0 \} \} \]
is nonempty. 
If $(u,x),(v,y) \in M$, then we write
$(u,x) \leq (v,y)$ if $x$ has a representation
in $L_R(E)$ of real degree less than or
equal to all real degrees of representations
of $y$ in $L_R(E)$. We write $(u,x) = (v,y)$
whenever $(u,x) \leq (v,y)$ and 
$(v,y) \leq (u,x)$. Clearly, clearly,
$\leq$ is a total order on $M$ which
therefore has a minimal element $(u,x)$.
Choose a minimizing representation
$x = \sum_{i=1}^n e_i a_i + b$
where $e_1,\ldots,e_n \in E^1$ are all distinct, 
each $a_i \in L_R(E)$ is either zero,
or nonzero and 
representable as an element of smaller 
real degree than that of $x$,
and $b$
is a polynomial (possibly zero) 
in only ghost paths whose source and range equals $u$. 
Take $i\in \{1,\ldots,n\}$.
Write $v_i := r(e_i)$.
By Lemma~\ref{lem:Zinvariant}, 
$e_i^* x e_i \in (Z(v_i L_R(E) v_i))_N$. Since
$e_i^* x e_i$ is of 
smaller real
degree 
than 
that of $x$, it 
follows that 
$e_i^* x e_i = 0$. 
Using that 
$x \in (Z(u L_R(E) u))_N$, it follows that
$e_i^* x = e_i^* e_i e_i^* x = e_i^* x e_i e_i^* = 0$.
Thus, 
$0 = e_i^* x = a_i + e_i^* b$
and hence
$a_i = -e_i^* b$.

Now, $0 \neq x = (u - \sum_{i=1}^n e_i e_i^*)b$. Thus, $u \neq \sum_{i=1}^n e_i e_i^*$ and $b\neq 0$.
This implies that there 
is some $f \in E^1 \setminus \{e_1,\ldots,e_n\}$ with $s(f) = u$. 
Furthermore, $f^* x = f^* b$,
and, by Lemma~\ref{lem:independent}, $f^* b \neq 0$ since it is a sum of distinct ghost paths.
Write $v:=r(f)$. By Lemma~\ref{lem:Zinvariant},
it follows that $f^* x f \in (Z(v L_R(E) v))_N$.
Using that $0 \neq f^* x = f^* f f^* x = f^* x f f^*$, we get that $f^* x f \neq 0$.
Note that the real degree of $f^* x f$ is less or equal to the real degree of $x$.
Hence, by the assumption made on $(u,x)$, 
and possibly after replacing
$(u,x)$ by $(v,f^*xf)$, we may assume that $a_i=0$ for every $i\in \{1,\ldots,n\}$. 
Therefore, suppose that $x = \sum_{j=1}^m r_j \beta_j^*$ for some 
nonzero $r_j \in R$ and some distinct paths 
$\beta_j \in E^{-N}$ with $s(\beta_j)=r(\beta_j) = u$.
Take $k \in \{ 1,\ldots,m \}$. 
By Lemma~\ref{lem:Zinvariant} it follows that
$r_k \beta_k^* = 
\beta_k^* x \beta_k 
\in Z(u L_R(E) u)$.
By assumption, the cycle $\beta_k$ has an
exit at some $w \in E^0$. Thus, there are
$\gamma,\delta \in E^*$ and $\epsilon \in E^1$ such that 
$\beta_k = \gamma \delta$, $r(\gamma) = s(\epsilon)$
and $\epsilon^* \delta = 0$.
By Lemma~\ref{lem:Zinvariant}, it follows 
that $r_k (\delta \gamma)^*
= r_k \gamma^* \delta^* \gamma^*\gamma
= \gamma^* r_k \beta_k^* \gamma
\in Z(v L_R(E) v)$.
We now reach a contradiction, because 
$0 \neq  \epsilon \epsilon^* r_k (\delta \gamma)^* = 
r_k (\delta \gamma)^* \epsilon \epsilon^* = 0$.
\end{proof}

Now, we prove our main result.

\begin{proof}[Proof of Theorem~\ref{thm:main}]
First we show the ``only if'' statement.
Suppose that $L_R(E)$ is simple. Then
$L_R(E)$ is graded simple and hence, by
Proposition~\ref{prop:GradedSimple}, it
follows that $R$ is simple and that
$E^0$ has no nontrivial hereditary and saturated subset.
Furthermore, Proposition~\ref{prop:simplicitylocal} implies that
$uL_R(E)u$ is simple for every $u \in E^0$, and hence,
by Proposition~\ref{prop:simplicity},  $Z(uL_R(E)u) \subseteq (uL_R(E)u)_0$
for every $u \in E^0$. Thus, by Proposition~\ref{prop:center},  every
cycle in $E$ has an exit.

Now we show the ``if'' statement.
Suppose that $R$ is simple, $E^0$ has no
nontrivial hereditary and saturated 
subset, and every cycle in $E$ has an exit.
By Proposition~\ref{prop:GradedSimple}, $L_R(E)$ is graded simple. 
Take $u \in E^0$.
It follows from Proposition~\ref{prop:center} that
$Z(uL_R(E)u) \subseteq (uL_R(E)u)_0$.
Furthermore, by Proposition~\ref{prop:simplicitylocal}, $uL_R(E)u$ is graded simple. Thus, by
Proposition~\ref{prop:simplicity} we get
that $uL_R(E)u$ is simple. Hence, by
Proposition~\ref{prop:simplicitylocalagain},
$L_R(E)$ is simple.
\end{proof}

\section{The center of a simple Leavitt path algebra}

In this section we prove Theorem~\ref{thm:center} using results from the previous sections together with some auxiliary observations.

\begin{remark}
Let $E=(E^0,E^1,r,s)$ be a directed graph.

(a)
Take $v \in E^0$. 
We write $w \leq v$, 
for $w \in E^0$,
if there is $\mu \in E^*$ with $s(\mu)=v$ and
$r(\mu)=w$.
 The set $T(v) := 
\{ w \in E^0 \mid w \leq v \}$ is the 
smallest hereditary subset of $E^0$
containing $v$. 

(b)
Suppose that $X \subseteq E^0$.
Put $T(X) := \cup_{x \in X} T(x)$.
The \emph{hereditary saturated closure}
$\overline{X}$ of $X$
is defined as the smallest hereditary 
and saturated subset of $E^0$ containing
$X$. One can show (see 
\cite[p.~626]{APCrow2011} and the references
therein) that
$\overline{X} = \cup_{n=0}^{\infty} X_n$
where $X_0 := T(X)$ and, for $n \geq 1$,
$X_n := \{ y \in E^0 \mid 0 < |s^{-1}(y)| <
\infty \ \mbox{and} \ 
r( s^{-1}(y) ) \subseteq X_{n-1} \} \cup
X_{n-1}$.
\end{remark}

The following result can be proved by induction (see \cite[Prop.~14.11]{primeness}) and
\cite[Lem.~5.2]{Tomforde2011}).

\begin{prop}\label{prop:tomforde}
Suppose that $R$ is an associative unital ring 
and that $E=(E^0,E^1,r,s)$ is a directed graph.
If $a \in (L_R(E))_0$ is nonzero, then there exist $\alpha, \beta \in E^*$, $v \in E^0$ and a nonzero
$k \in R$ such that $\alpha^* a \beta = kv$.
\end{prop}

Now, we prove our second main result.

\begin{proof}[Proof of Theorem~\ref{thm:center}]
Write $S := L_R(E)$. If $S$ is not unital,
then it follows immediately from
\cite[3.3]{Wisbauer1991} that $Z(S)=\{0\}$.
This proves (a). Now, we show (b).
Suppose that $S$ is unital, i.e. $E^0$ is finite.
Take a nonzero $x \in Z(S)$.
By Proposition~\ref{prop:simplicity}, 
it follows that $x \in S_0$. Therefore,
by Proposition~\ref{prop:tomforde}, 
there are $\alpha,\beta \in E^*$, 
$v \in E^0$ and a nonzero $k \in R$ 
such that $\alpha^* x \beta = kv$.
From this equality, the grading, and the fact that 
$x \in Z(S)$, it follows that $\alpha=\beta$ and $r(\alpha)=v$.
Hence,
$vx = \alpha^*\alpha x = \alpha^* x\alpha= 
\alpha^* x \beta = kv$.
Note that the equality $vx = kv$ 
implies that $k \in Z(R)$.
Put $X := \{ v \}$. Then $\overline{X}$ is a
nonempty
hereditary and saturated subset of $E^0$.
By Theorem~\ref{thm:main}, 
$\overline{X} = E^0$. 
We claim that this implies that
$wx = kw$ for every $w \in E^0$.
Let us assume, for a moment, that this
claim holds. Then $x = 1_S \cdot x = 
\sum_{w \in E^0} wx = \sum_{w \in E^0} kw =
k \cdot \sum_{w \in E^0} w = k \cdot 1_S
\in Z(R) \cdot 1_S$. Thus, $Z(S) \subseteq
Z(R) \cdot 1_S$. Clearly, $Z(R) \cdot 1_S
\subseteq Z(S)$ holds.

Now we show the claim. We will use induction
to prove that for every $n \geq 0$ the 
implication $w \in X_n \Rightarrow wx = kw$
holds. From this the claim follows.
Base case: $n=0$. 
Suppose that $w\in X_0$, i.e. $w \leq v$. Then there is 
a path $\delta$ from $v$ to $w$.
We get that $wx = \delta^* \delta x =
\delta^* v \delta x = \delta^* vx \delta =
\delta^* kv \delta = k \delta^* v \delta =
k \delta^* \delta = k w$.
Induction step:
Suppose that $wx = kw$ for every 
$w \in X_{n-1}$. Take $y \in X_n \setminus 
X_{n-1}$
and note that
$0 < |s^{-1}(y)| <
\infty$ and $r( s^{-1}(y) ) \subseteq X_{n-1}$.
We get that $yx = \sum_{e \in s^{-1}(y)}
e e^* x = \sum_{e \in s^{-1}(y)} e r(e) x e^*
= \sum_{e \in s^{-1}(y)} e k r(e) e^* =
k \sum_{e \in s^{-1}(y)} e e^* = ky$.
\end{proof}

\end{document}